\documentclass[11pt,a4paper]{article}
\usepackage[T1]{fontenc}
\usepackage{lmodern}
\usepackage[margin=27mm]{geometry}
\usepackage{amsmath,amssymb,amsthm,mathtools}
\usepackage{microtype}
\usepackage{xcolor}
\usepackage{titlesec}
\usepackage[colorlinks=true,linkcolor=blue!45!black,citecolor=blue!45!black,urlcolor=blue!45!black]{hyperref}
\newtheorem{lemma}{Lemma}
\newtheorem{prop}{Proposition}

\newtheorem{corollary}[lemma]{Corollary}
\newtheorem{theorem}[lemma]{Theorem}
\theoremstyle{definition}
\newtheorem{defn}{Definition}
\theoremstyle{remark}
\newtheorem{rem}{Remark}
\DeclareMathOperator{\Ree}{Re}

\newcommand{\BH}{\mathcal B(H)}
\newcommand{\HC}{H_{\mathbb C}}
\newcommand{\BHC}{\mathcal B(\HC)}
\newcommand{\norm}[1]{\left\lVert#1\right\rVert}

\hypersetup{pdftitle={Unique operator space }}
\title{Unique operator space structures on a family of two dimensional Banach spaces}
\author{Samya Kumar Ray}
\begin{document}
\maketitle

\begin{abstract}
For $0<\theta<\pi$, let $K_\theta=\{e^{is}:|s|\le\theta\}$ and
$E_\theta=\operatorname{span}_{\mathbb C}\{1,z\}\subset C(K_\theta)$.
We prove that every contractive linear map from $E_\theta$ into
$\mathcal B(H)$ admits a factorization through a unitary operator whose
spectrum is contained in $K_\theta$, and is therefore completely
contractive. Consequently, $E_\theta$ has a unique operator space
structure. For $0<\theta,\varphi<\pi$, we also show that
$E_\theta$ and $E_\varphi$ are isometrically isomorphic if and only if
$\theta=\varphi$, and that none
of these spaces is isometrically isomorphic to $\ell_1^2$ or
$\ell_\infty^2$. Thus we obtain uncountably many mutually
non-isometric two-dimensional complex Banach spaces with a unique
operator space structure, giving a negative answer to a question of Pisier.
\end{abstract}
\titleformat{\section}[block]{\normalfont\Large\bfseries\filcenter}{\thesection}{1em}{}
\titleformat{\subsection}[block]{\normalfont\large\bfseries\filcenter}{\thesubsection}{1em}{}
\section{Introduction}\label{sec:introduction}
Let $\mathcal D\subseteq\mathbb C^m$ be the open unit ball of a
norm $\|\cdot\|_{\mathcal D}$, and let $M_n$ be equipped with the
usual operator norm. Denote by $\mathcal O(\overline{\mathcal D})$
the closure of the polynomials on $\overline{\mathcal D}$ in the
supremum norm. For $w\in\mathcal D$ and an $m$-tuple
$\mathbf T=(T_1,\ldots,T_m)$ of matrices in $M_n$, set
\[
\langle\Delta f(w),\mathbf T\rangle
:=
\sum_{j=1}^m
\frac{\partial f}{\partial z_j}(w)T_j,
\qquad\text{for all }f\in\mathcal O(\overline{\mathcal D}).
\]
Then
\begin{equation}\label{PARHOM}
\Phi_{(w,\mathbf T)}(f)
:=
\begin{pmatrix}
f(w)I_n & \langle\Delta f(w),\mathbf T\rangle\\
0 & f(w)I_n
\end{pmatrix}
\end{equation}
defines a continuous unital algebra homomorphism from
$\mathcal O(\overline{\mathcal D})$ into $M_{2n}$. We call such
a map a \emph{Parrott-like homomorphism}. These homomorphisms
generalize the construction underlying Parrott's example
\cite{PA70}; their contractivity and complete contractivity were
studied systematically by Bagchi and Misra \cite{BM95}.

A basic question is to determine which domains admit a
contractive homomorphism that is not completely contractive.
This question is closely related to the existence of boundary
normal dilations. Indeed Arveson's dilation theorem characterizes
complete contractivity of unital representations of function
algebras by the existence of dilations supported on the
Shilov boundary \cite[Theorem~1.2.2 and its corollary]{Arveson}.

For homomorphisms based at the origin, the question can be
expressed in terms of the operator space structure of the
underlying Banach space. Namely, every contractive
Parrott-like homomorphism $\Phi_{(0,\mathbf T)}$, for every
matrix size $n$, is completely contractive if and only if
$(\mathbb C^m,\|\cdot\|_{\mathcal D})$ has a unique operator
space structure; see \cite{PA92} and
\cite[Theorem~1.1 and Lemma~1.4]{BM95}.
For a complex Banach space $E$, put
\[
\alpha(E)
:=
\bigl\|\operatorname{id}:\operatorname{MIN}(E)
\longrightarrow\operatorname{MAX}(E)\bigr\|_{\mathrm{cb}},
\]
where $\operatorname{MIN}(E)$ and $\operatorname{MAX}(E)$
denote its minimal and maximal operator space structures,
recalled in Section~\ref{sec:operator-spaces}.
We say $E$ has a unique operator space structure precisely when
$\alpha(E)=1$.

Using results from Banach space geometry, Paulsen \cite{PA92}
proved that every complex Banach space of dimension at least
five admits different operator space structures. In fact, no
complex Banach space of dimension at least three has a unique
operator space structure, whereas $\ell_1^2$ and
$\ell_\infty^2$ do have this property
\cite[Corollary~3.9 and Proposition~3.10]{Pisier}.
Pisier posed the following question
\cite[Chapter~3, Open problems, Problem~1, p.~79]{Pisier}:
\begin{quote}
``Are the spaces $\mathbb C$, $\ell_1^2$, and $\ell_\infty^2$
the only operator spaces $E\ne\{0\}$ admitting a unique
operator space structure, that is, for which $\alpha(E)=1$?''
\end{quote}

In this paper, we give a negative answer to this question by
constructing uncountably many mutually non-isometric
two-dimensional complex Banach spaces with a unique operator
space structure, none of which is isometrically isomorphic to
$\ell_1^2$ or $\ell_\infty^2$.
Our choice of examples was motivated by \cite[Corollary 2.8]{ARFJS95} and \cite[Theorem 1.4]{Ray20}. We consider the spaces $E_\theta$.
For $0<\theta<\pi$, consider
\[
K_\theta=\{e^{is}:-\theta\le s\le\theta\},
\qquad
E_\theta=\operatorname{span}_{\mathbb C}\{1,z\}
\subset C(K_\theta),
\]
where $z(\zeta)=\zeta$ for all $\zeta\in K_\theta$.
\begin{theorem}\label{thm:unique-oss}
Let $0<\theta<\pi$, and equip $E_\theta$ with the operator space
structure inherited from $C(K_\theta)$. Every contractive linear map
$u:E_\theta\to\mathcal B(H)$ is completely contractive, for every
complex Hilbert space $H$. Consequently, $E_\theta$ has a unique
operator space structure.
\end{theorem}
The proof of Theorem~\ref{thm:unique-oss} is given in
Section~\ref{sec:main-results} and uses Theorem~\ref{thm:arc}.
Corollary~\ref{cor:not-ell1} and Proposition~\ref{prop:not-ellinfty}
show that $E_\theta$ is not isometrically isomorphic to either
$\ell_1^2$ or $\ell_\infty^2$. Proposition~\ref{prop:pairwise-nonisometric}
shows that $E_\theta$ and $E_\varphi$ are isometrically isomorphic
if and only if $\theta=\varphi$.

Section~\ref{sec:operator-spaces} gives the definitions of operator
spaces and their minimal and maximal structures.
Section~\ref{sec:main-results} states Lemma~\ref{lem:positive-block}
and proves Lemma~\ref{lem:completion}, Corollary~\ref{cor:gram},
and Theorems~\ref{thm:arc} and~\ref{thm:unique-oss}.
Section~\ref{sec:nonisometry} proves that the spaces $E_\theta$
are mutually non-isometric and are not isometric to $\ell_1^2$
or $\ell_\infty^2$.
Appendices~\ref{app:complexification} and~\ref{app:realification}
provides some basic results on complexification and realification.
Appendices~\ref{app:complexification} and~\ref{app:realification}
are included because the author could not find a suitable reference
where these results are collected.

\section{Operator spaces}\label{sec:operator-spaces}

We denote by $M_{m,n}$ the space of complex $m\times n$ matrices,
and write $M_n=M_{n,n}$.

\begin{defn}[Abstract operator space]\label{def:abstract-operator-space}
An abstract operator space is a normed complex linear space $X$
together with norms $\|\cdot\|_n$ on $M_n(X)$ for all $n\ge1$,
satisfying the following conditions:
\begin{itemize}
\item The norm $\|\cdot\|_1$ on $M_1(X)=X$ agrees with the given
norm of $X$.
\item For all $A\in M_n(X)$ and $B\in M_m(X)$,
\[
\|A\oplus B\|_{n+m}=\max\{\|A\|_n,\|B\|_m\},
\]
where $A\oplus B$ is the block diagonal array with diagonal
blocks $A$ and $B$.
\item For all $A\in M_{n,m}$, $B\in M_{m,n}$ and $C\in M_m(X)$,
\[
\|ACB\|_n
\le
\|A\|_{\ell_2^m\to\ell_2^n}
\|C\|_m
\|B\|_{\ell_2^n\to\ell_2^m}.
\]
\end{itemize}
\end{defn}

\begin{defn}[Concrete operator space]\label{def:concrete-operator-space}
A concrete operator space is a triple $(X,i,H)$, where
$X$ is a normed complex linear space, $H$ is a complex
Hilbert space, and $i:X\to\mathcal B(H)$ is a linear isometry.
\end{defn}

Let $(X,i,H)$ be a concrete operator space. For all $m,n\ge1$,
the map
\[
\mathrm{id}_{M_{m,n}}\otimes i:
M_{m,n}(X)\longrightarrow M_{m,n}(\mathcal B(H))
\]
applies $i$ to each entry. We identify
$M_{m,n}(\mathcal B(H))$ with
$\mathcal B(H^n,H^m)$, where $H^r$
is the Hilbert space direct sum of $r$ copies of $H$.
Under this identification, an array $(T_{ij})$ acts by
\[
(\xi_j)_{j=1}^n\longmapsto
\left(\sum_{j=1}^nT_{ij}\xi_j\right)_{i=1}^m
\qquad\text{for all }(\xi_j)_{j=1}^n\in H^n.
\]
For all $(x_{ij})\in M_{m,n}(X)$, define
\begin{equation}\label{eq:concrete-matrix-norm}
\|(x_{ij})\|_{m,n}
:=\bigl\|(i(x_{ij}))\bigr\|_{H^n\to H^m}.
\end{equation}
For $m=n$, we write $\|\cdot\|_n=\|\cdot\|_{n,n}$.
These norms satisfy the conditions in
Definition~\ref{def:abstract-operator-space}. The converse follows
from Ruan's theorem \cite{Ruan1988}.

\begin{theorem}[Ruan's theorem]\label{thm:ruan}
Let $X$ be an abstract operator space with matrix norms
$\{\|\cdot\|_n:n\ge1\}$. There are a complex Hilbert space
$H$ and a linear isometry $i:X\to\mathcal B(H)$
such that
\begin{equation}\label{eq:ruan-representation}
\|(x_{ij})\|_n
=\bigl\|(\mathrm{id}_{M_n}\otimes i)((x_{ij}))\bigr\|
_{H^n\to H^n}
\end{equation}
for all $n\ge1$ and all $(x_{ij})\in M_n(X)$.
\end{theorem}

\begin{rem}
By Theorem~\ref{thm:ruan}, we may identify an abstract operator
space with a concrete operator space. We shall use the term
operator space for either realization.
\end{rem}

Let $X,Y$ be operator spaces and let $u:X\to Y$ be linear.
For all $n\ge1$, the amplification $u^{(n)}:M_n(X)\to M_n(Y)$ is
\begin{equation}\label{eq:amplification}
u^{(n)}((x_{ij})):=(u(x_{ij}))
\qquad\text{for all }(x_{ij})\in M_n(X).
\end{equation}
The completely bounded norm of $u$ is
\[
\|u\|_{\mathrm{cb}}:=\sup_{n\ge1}\|u^{(n)}\|.
\]
We say that $u$ is completely bounded if
$\|u\|_{\mathrm{cb}}<\infty$, and completely contractive if
$\|u\|_{\mathrm{cb}}\le1$. A linear map is a complete isometry
if each of its amplifications is an isometry.

For a normed space $(X,\|\cdot\|_X)$, write
\[
(X,\|\cdot\|_X)_1:=\{x\in X:\|x\|_X\le1\}
\]
for its closed unit ball, and write $(X)_1$ when the norm is clear.
We denote by $X^*$ the space of bounded complex-linear
functionals on $X$, with its usual dual norm.

Let $X$ be a normed complex linear space. Its minimal operator
space structure is defined by
\begin{equation}\label{eq:min-matrix-norm}
\|(x_{ij})\|_{M_n(\operatorname{MIN}(X))}
:=\sup_{f\in(X^*)_1}
\bigl\|(f(x_{ij}))\bigr\|_{\ell_2^n\to\ell_2^n}
\end{equation}
for all $n\ge1$ and all $(x_{ij})\in M_n(X)$.
These norms satisfy Definition~\ref{def:abstract-operator-space}
and agree with the given norm at $n=1$; see
\cite{EffrosRuan2000} and \cite[Chapter~3]{Pisier}.
We denote the resulting operator space by $\operatorname{MIN}(X)$.

The maximal operator space structure on $X$ is defined by
\begin{equation}\label{eq:max-matrix-norm}
\|(x_{ij})\|_{M_n(\operatorname{MAX}(X))}
:=\sup\left\{
\bigl\|(T(x_{ij}))\bigr\|_{H^n\to H^n}:
T:X\to\mathcal B(H),\ \|T\|\le1
\right\}
\end{equation}
for all $n\ge1$ and all $(x_{ij})\in M_n(X)$.
Here the supremum is over all complex Hilbert spaces $H$
and all contractive linear maps $T:X\to\mathcal B(H)$.
These norms also satisfy Definition~\ref{def:abstract-operator-space}
and agree with the norm of $X$ at $n=1$; see
\cite{EffrosRuan2000} and \cite[Chapter~3]{Pisier}.
We write $\operatorname{MAX}(X)$ for this operator space.

For any operator space structure on $X$ inducing its given norm,
\[
\|(x_{ij})\|_{M_n(\operatorname{MIN}(X))}
\le\|(x_{ij})\|_{M_n(X)}
\le\|(x_{ij})\|_{M_n(\operatorname{MAX}(X))}
\]
for all $n\ge1$ and all $(x_{ij})\in M_n(X)$; see
\cite[Chapter~3]{Pisier}.
Consequently, $X$ has a unique operator space structure if and only if
these minimal and maximal norms agree at every matrix level.
By \eqref{eq:max-matrix-norm}, this is equivalent to saying that
every contractive linear map
$\operatorname{MIN}(X)\to\mathcal B(H)$ is completely
contractive, for every complex Hilbert space $H$.
For a compact Hausdorff space $K$ and a subspace $X\subset C(K)$
with its supremum norm, the matrix
norms inherited from $C(K)$ are those of $\operatorname{MIN}(X)$
\cite[Chapter~3]{Pisier}.
We refer to \cite{EffrosRuan2000,Pisier} for further discussion.

\section{Main results}\label{sec:main-results}
Unless specified, we consider only Hilbert spaces over complex field
and complex-linear operators. Denote $\Ree T=(T+T^*)/2$. Let $\mathbb{R}[x]$ denote all polynomials with coefficients in $\mathbb R.$ We denote $\mathcal{B}(H)_{\text{h}}$ to be the Banach space of all bounded self-adjoint
linear operators on $H.$

For $0<\theta<\pi$, let us put
\[
K_\theta=\{e^{is}:-\theta\le s\le\theta\},\qquad
E_\theta=\operatorname{span}_{\mathbb C}\{1,z\}\subset C(K_\theta),
\qquad c=\cos\theta,
\]
where $z:K_\theta\to\mathbb C$ is given by $z(\zeta)=\zeta.$

We begin with the following result \cite[Proposition~3(2)]{CZ}.
\begin{prop}\label{prop:operator Riesz type}
Let $H$ be a real Hilbert space and let
$F\in\mathbb R[x]\otimes\mathcal{B}(H)_{\mathrm{h}}$. Then
$F(t)\geq 0$ for all $t\ge 0$ if and only if
\[
  F(t)=R_0+t\,R_1,
\]
where $R_0,R_1$ are sums of Hermitian squares of polynomials
from $\mathbb R[x]\otimes\mathcal{B}(H)$, that is for some $(U_i)_{i=1}^n,(V_j)_{j=1}^m\subseteq \mathbb{R}[x]\otimes\mathcal{B}(H),$ $R_0=\sum_{i=1}^nU_i^*U_i$ and $R_1=\sum_{j=1}^mV_j^*V_j.$
\end{prop}

\begin{samepage}
We shall also use the following standard result; see \cite[Theorem~5.11]{MKX19} for the complex case.
\begin{lemma}[Positive-block factorization]\label{lem:positive-block}
Let $H_1,H_2$ be real Hilbert spaces or complex Hilbert spaces. Let $P\in\mathcal B(H_1)$ and $Q\in\mathcal B(H_2)$ be positive,
and let $X\in\mathcal B(H_2,H_1)$. Then
\[
\begin{pmatrix}P&X\\X^*&Q\end{pmatrix}\ge0
\]
if and only if there is a contraction $K\in\mathcal B(H_2,H_1)$ such that
\[
X=P^{1/2}KQ^{1/2}.
\]
\end{lemma}
\end{samepage}

\begin{lemma}\label{lem:completion}
Let $H$ be a real or complex Hilbert space, let $P,Q\in\BH$ be positive,
and let $S=S^*\in\BH$. The following are equivalent:
\begin{enumerate}
\item[(i)]
\begin{equation}\label{eq:pencil}
P+tS+t^2Q\ge0\qquad\text{for all}\ t\ge0.
\end{equation}
\item[(ii)] There is a contraction $K\in\BH$ such that
\begin{equation}\label{eq:coupling}
P^{1/2}KQ^{1/2}+Q^{1/2}K^*P^{1/2}\le S.
\end{equation}
\end{enumerate}
Equivalently, there is an operator $X\in\BH$ satisfying
\begin{equation}\label{eq:block}
\begin{pmatrix}P&X\\X^*&Q\end{pmatrix}\ge0,
\qquad X+X^*\le S.
\end{equation}
\end{lemma}

\begin{proof}
We first solve the real case.

By Proposition \ref{prop:operator Riesz type} there are finitely many operator polynomials
\(U_i,V_j\) such that
\begin{equation}\label{eq:quadratidecomp}
P+tS+t^2Q
=\sum_i U_i(t)^*U_i(t)+t\sum_jV_j(t)^*V_j(t), \ \text{for all}\ t\geq 0.
\end{equation}
Note that $\sum U_i^*U_i$ has even degree and $t\mapsto t\sum V_j(t)^*V_j(t)$ has odd degree. It follows that
\[
U_i(t)=A_i+tB_i,\qquad V_j(t)=C_j.
\]
for some $A_i,B_i,C_j\in \mathcal{B}(H).$
Comparing coefficients in \eqref{eq:quadratidecomp} yields
\[
P=\sum_i A_i^*A_i,\qquad Q=\sum_i B_i^*B_i,\qquad\text{and}\ S=X+X^*+Z,
\]
where $X=\sum_iA_i^*B_i$ and $Z=\sum_jC_j^*C_j\ge0.$
Consequently,
\begin{equation}\label{eq:real-gram}
\begin{pmatrix}P&X\\X^*&Q\end{pmatrix}
=\sum_i
\begin{pmatrix}A_i^*\\B_i^*\end{pmatrix}
\begin{pmatrix}A_i&B_i\end{pmatrix}\ge0,
\qquad X+X^*\le S.
\end{equation}

By Lemma~\ref{lem:positive-block} we obtain \(X=P^{1/2}K_0Q^{1/2}\) for a real-linear contraction \(K_0\). This shows $(i)$ implies $(ii).$

\medskip
We now deal with the complex case.
Now let \(H\) be complex, and regard it as the real Hilbert space
\(H_{\mathbb R}\) with inner product
\(\langle x,y\rangle_{\mathbb R}=\Ree\langle x,y\rangle\) as in Appendix~\ref{app:realification}. Identifying $P,Q$ and $S$ as bounded linear operators on $H_{\mathbb R},$ they are again positive and self-adjoint respectively. Moreover, the operators $P,Q$ and $S$ therefore satisfy the same assumptions on
\(H_{\mathbb R}\), with the same square roots \(P^{1/2},Q^{1/2}\).
These identifications follow from Proposition~\ref{prop:realification} and the square-root identity immediately following it.
The real result supplies a real-linear
contraction \(K_0\) with
\begin{equation}\label{eq:real-coupling}
P^{1/2}K_0Q^{1/2}+Q^{1/2}K_0^*P^{1/2}\le S.
\end{equation}
In above $K_0^*$ is adjoint as an operator on the real-Hilbert space \(H_{\mathbb R}\).
 Set
\[
K=\frac12(K_0+J^*K_0J)
=\frac12(K_0-JK_0J).
\] where $J$ is as in Appendix~\ref{app:realification}. By Proposition~\ref{prop:realification}, \(K\) is complex-linear and a contraction. Since \(P,Q,S,P^{1/2},Q^{1/2}\)
commute with \(J\), conjugating \eqref{eq:real-coupling} by \(J\),
adding the resulting inequality to \eqref{eq:real-coupling}, and dividing by two gives
\[
P^{1/2}KQ^{1/2}+Q^{1/2}K^*P^{1/2}\le S.
\]
Since \(K\) is now complex-linear, this is precisely \eqref{eq:coupling}
on \(H\). This proves $(i)$ implies $(ii)$ in the complex case. Now we verify the equivalence of \eqref{eq:coupling} and \eqref{eq:block}, using Lemma~\ref{lem:positive-block}.
For any contraction \(K\), putting \(X=P^{1/2}KQ^{1/2}\) gives
\[
\begin{pmatrix}P&X\\X^*&Q\end{pmatrix}
=
\begin{pmatrix}P^{1/2}&0\\0&Q^{1/2}\end{pmatrix}
\begin{pmatrix}I&K\\K^*&I\end{pmatrix}
\begin{pmatrix}P^{1/2}&0\\0&Q^{1/2}\end{pmatrix}\ge0.
\]
If \eqref{eq:coupling} holds, then \(X+X^*\le S\).
Conversely, if \eqref{eq:block} holds, Lemma~\ref{lem:positive-block} gives a contraction $K$ such that $X=P^{1/2}KQ^{1/2}$, and $X+X^*\le S$ is precisely \eqref{eq:coupling}. Thus \eqref{eq:coupling} and \eqref{eq:block} are equivalent.

To prove $(ii)\Rightarrow(i)$, we can therefore use \eqref{eq:block}.
Testing the positive block on \((h,th)\) and adding
\(t\langle(S-X-X^*)h,h\rangle\ge0\) proves \eqref{eq:pencil}
for every \(t\ge0\). 
\end{proof}

\begin{corollary}\label{cor:gram}
Suppose $A,B\in\BH$ and $c\in[-1,1]$ satisfy
\begin{equation}\label{eq:test}
\norm{A-tB}^2\le 1+t^2-2ct\qquad\text{for all }t\ge0.
\end{equation}
There is a contraction $T\in\BH$ with $\Ree T\ge cI$ such that
\begin{equation}\label{eq:gram}
\begin{pmatrix}I&T\\T^*&I\end{pmatrix}
\ge
\begin{pmatrix}A^*A&A^*B\\B^*A&B^*B\end{pmatrix}.
\end{equation}
\end{corollary}

\begin{proof}
Taking $t=0$ and letting $t\to\infty$ in \eqref{eq:test} shows that $A,B$
are contractions. Set
\[
D_A=(I-A^*A)^{1/2},\qquad D_B=(I-B^*B)^{1/2}.
\]
Using $(A-tB)^*(A-tB)\leq \|A-tB\|^2$ we get from \eqref{eq:test}
\begin{equation}\label{eq:arc-pencil}
D_A^2+t(A^*B+B^*A-2cI)+t^2D_B^2\ge0\qquad\text{for all }t\ge0.
\end{equation}
Apply Lemma~\ref{lem:completion} with
$P=D_A^2$, $Q=D_B^2$, and $S=A^*B+B^*A-2cI$. It gives a contraction $K$
satisfying
\begin{equation}\label{eq:arc-coupling}
D_AKD_B+D_BK^*D_A\le A^*B+B^*A-2cI.
\end{equation}
Put $T=A^*B-D_AKD_B$. Then $\text{Re}T=\frac{T+T^*}{2}\ge cI$, and
\begin{align*}
&\begin{pmatrix}I&T\\T^*&I\end{pmatrix}
-\begin{pmatrix}A^*A&A^*B\\B^*A&B^*B\end{pmatrix}\\
&\qquad=
\begin{pmatrix}D_A&0\\0&D_B\end{pmatrix}
\begin{pmatrix}I&-K\\-K^*&I\end{pmatrix}
\begin{pmatrix}D_A&0\\0&D_B\end{pmatrix}\ge 0,
\end{align*}
as $\begin{psmallmatrix}I&-K\\-K^*&I\end{psmallmatrix}\ge0$ since $K$ is a contraction. Moreover, the block operator on the right of \eqref{eq:gram} is positive implying $T$ is a contraction.
\end{proof}
\subsection{Dilation and complete contractivity on the arc}\label{sec:arc}
\begin{theorem}\label{thm:arc}
Let $0<\theta<\pi$, let $H$ be a complex Hilbert space, and let
$u:E_\theta\to\mathcal B(H)$ be linear. Write $A=u(1)$ and $B=u(z)$.
The following are equivalent:
\begin{enumerate}
\item $u$ is contractive.
\item $\norm{A-tB}^2\le1+t^2-2(\cos\theta)t$ for every $t\ge0$.
\item There are a Hilbert space $L$, an isometry $V:H\to L$,
a contraction $R:L\to H$, and a unitary $U\in\mathcal B(L)$
with $\sigma(U)\subset K_\theta$ such that
\begin{equation}\label{eq:factor}
u(a+bz)=R(aI_{L}+bU)V\qquad\text{for all }a,b\in\mathbb C.
\end{equation}
\end{enumerate}
\end{theorem}
\begin{proof}
Put $c=\cos\theta$. Since the minimum of $\Ree\zeta$ on $K_\theta$ is $c$,
\begin{equation}\label{eq:scalar-test}
\norm{1-tz}_{C(K_\theta)}^2
=\max_{\zeta\in K_\theta}(1+t^2-2t\Ree\zeta)
=1+t^2-2ct\qquad\text{for all}\ t\ge0.
\end{equation}
Thus $(1)\Rightarrow(2)$.

Assume $(2)$. Corollary~\ref{cor:gram} gives a contraction $T$ satisfying
\eqref{eq:gram} and $\Ree T\ge cI$. By Choi--Li
\cite[Theorem 2.1, p.~437]{CL}, there are an isometry $V:H\to H\oplus H$ and
a unitary $U\in\mathcal B(H\oplus H)$ such that
\[
T=V^*UV,\qquad \Ree U\ge cI.
\]
 Since $U$ is unitary,
\begin{equation}\label{eq:spectrum}
\sigma(U)\subset\{\zeta\in\mathbb T:\Ree\zeta\ge c\}=K_\theta.
\end{equation}
Define the block operators
\[
W=[\,V\ \ UV\,]:H\oplus H\to H\oplus H,\qquad
F=[\,A\ \ B\,]:H\oplus H\to H.
\]
Then we have
\[
W^*W=\begin{pmatrix}I&T\\T^*&I\end{pmatrix},\qquad
F^*F=\begin{pmatrix}A^*A&A^*B\\B^*A&B^*B\end{pmatrix},
\]
so \eqref{eq:gram} says $F^*F\le W^*W$.
Douglas's theorem \cite{Douglas}
gives $F=RW$ for a contraction $R:H\oplus H\to H$.
Thus $A=RV$ and $B=RUV$, proving $(3)$.

Assume $(3)$. We obtain by (3)
\begin{align}
\|u(a+bz)\|=\|R(aI+bU)V\|\leq \|aI+bU\|\leq\sup_{\zeta\in K_\theta}|a+b\zeta|=\|(a+zb)\|_{E_\theta}.
\label{eq:matrix}
\end{align}
In above we use functional calculus of normal operators. This shows $u$ is a contraction.
\end{proof}
\begin{proof}[Proof of Theorem~\ref{thm:unique-oss}]
Let $H$ be a complex Hilbert space and let
$u:E_\theta\to\mathcal B(H)$ be contractive.
The assertion is immediate if $H=\{0\}$, so assume $H\ne\{0\}$.
By Theorem~\ref{thm:arc}, there are a complex Hilbert space $L$,
an isometry $V:H\to L$, a contraction $R:L\to H$, and a unitary
$U\in\mathcal B(L)$ with $\sigma(U)\subseteq K_\theta$ such that
\[
u(a+bz)=R(aI_L+bU)V
\qquad\text{for all }a,b\in\mathbb C.
\]
Fix $n\ge1$. Every element of $M_n(E_\theta)$ has the form $C+Dz$,
where $C=(c_{ij}),D=(d_{ij})\in M_n$ and
$C+Dz=(c_{ij}1+d_{ij}z)_{i,j=1}^n$.
Applying $u$ to each entry gives
\[
u^{(n)}(C+Dz)
=(I_n\otimes R)(C\otimes I_L+D\otimes U)(I_n\otimes V).
\]
Since $R$ and $V$ are contractions,
\[
\|u^{(n)}(C+Dz)\|
\le\|C\otimes I_L+D\otimes U\|.
\]
We now justify the estimate for the operator on the right.
For every $\varepsilon>0$, the spectral theorem gives finitely many
mutually orthogonal spectral projections $P_1,\ldots,P_m$ with
$\sum_{j=1}^mP_j=I_L$, and points
$\zeta_1,\ldots,\zeta_m\in\sigma(U)$, such that
\[
U_\varepsilon=\sum_{j=1}^m\zeta_jP_j,
\qquad \|U-U_\varepsilon\|\le\varepsilon.
\]
The operator
\[
C\otimes I_L+D\otimes U_\varepsilon
=\sum_{j=1}^m(C+\zeta_jD)\otimes P_j
\]
satisfies the inequality
\[
\|C\otimes I_L+D\otimes U_\varepsilon\|
\le\max_{1\le j\le m}\|C+\zeta_jD\|
\le\sup_{\zeta\in K_\theta}\|C+\zeta D\|.
\]

Letting $\varepsilon\to0$ we obtain
\[
\|u^{(n)}(C+Dz)\|
\le\sup_{\zeta\in K_\theta}\|C+\zeta D\|
=\|C+Dz\|_{M_n(E_\theta)}.
\]
Thus $u^{(n)}$ is contractive for every $n\ge1$, so $u$ is
completely contractive. The operator space structure inherited from
$C(K_\theta)$ is minimal. Since every contractive map from this
space into $\mathcal B(H)$ is completely contractive,
\eqref{eq:max-matrix-norm} shows that its maximal matrix norms are
at most its minimal matrix norms. This completes the proof.
\end{proof}

\section{\texorpdfstring{$E_\theta$ is not isometric to $\ell_\infty^2$ or $\ell_1^2$}{E(theta) is not isometric to ell-infinity(2) or ell-1(2)}}\label{sec:nonisometry}

\begin{lemma}
\label{lem:real-restriction}
Let $X$ be a two-dimensional complex Banach space, and let $x,y\in X$
be complex-linearly independent. Suppose that
\[
\|x+y\|=\|x\|+\|y\|.
\]
If $X$ is complex-linearly isometrically isomorphic to
$\ell_1^2(\mathbb C)$, then $F=\operatorname{span}_{\mathbb R}\{x,y\},$
with norm induced from $X$, is real-linearly isometrically isomorphic to
$\ell_1^2(\mathbb R)$.
\end{lemma}

\begin{proof}
Let $T:X\to\ell_1^2(\mathbb C)$ be a complex-linear isometric
isomorphism, and write
\[
u=Tx=(u_1,u_2),\qquad v=Ty=(v_1,v_2).
\]
Since $\{x,y\}$ forms a complex basis of $X$, the vectors $\{u,v\}$ forms a
complex basis of $\mathbb C^2$. In particular,
$(u_j,v_j)\ne(0,0)$ for each $j$. Now the hypothesis gives
\[
\sum_{j=1}^2 |u_j+v_j|
=\|u+v\|_1
=\|u\|_1+\|v\|_1
=\sum_{j=1}^2 (|u_j|+|v_j|).
\]
Then there are $\omega_j\in\mathbb T$
and $r_j,s_j\ge0$ such that $u_j=r_j\omega_j,$ and  $v_j=s_j\omega_j.$

Define $D(w_1,w_2)
=(\overline{\omega_1}w_1,\overline{\omega_2}w_2).$
This is a complex-linear isometry of $\ell_1^2(\mathbb C)$.
Moreover,
for every $a,b\in\mathbb R$, $DT(ax+by) \in\mathbb R^2.$
Hence $DT$ maps $F$ into $\mathbb R^2$. Clearly, $(DT)|_F$ is a real-linear isometric isomorphism from $F$
onto $\ell_1^2(\mathbb R)$.
\end{proof}

\begin{corollary}\label{cor:not-ell1}
Let $0<\theta<\pi$. The real two-dimensional Banach space
\[
F_\theta=\operatorname{span}_{\mathbb R}\{1,z\}\subset E_\theta
\]
is not real-linearly isometrically isomorphic to
$\ell_1^2(\mathbb R)$. Consequently, $E_\theta$ is not
complex-linearly isometrically isomorphic to $\ell_1^2(\mathbb C)$.
\end{corollary}

\begin{proof}
Define for all $t>0$,  $q_\theta(t):=\|1-tz\|
=\sqrt{1+t^2-2t\cos\theta}.$
Then
\[
\begin{aligned}
q_\theta''(t)=\frac{\sin^2\theta}
        {(1+t^2-2t\cos\theta)^{3/2}}
>0.
\end{aligned}
\]

Suppose that
$V:F_\theta\to\ell_1^2(\mathbb R)$ were a real-linear isometric
isomorphism. Write
\[
V1=(r_1,r_2),\qquad Vz=(s_1,s_2).
\]
Then
\[
q_\theta(t)
=\|V(1-tz)\|_1
=|r_1-ts_1|+|r_2-ts_2|,
\qquad\text{for all }t>0.
\]
The right-hand side is piecewise affine. This contradicts
$q_\theta''(t)>0$ for all $t>0.$

Finally, $1\in K_\theta$, so
\[
\|1+z\|=2=\|1\|+\|z\|.
\]
Since $1,z$ are complex-linearly independent,
Lemma~\ref{lem:real-restriction} shows that a complex-linear
isometric isomorphism $E_\theta\to\ell_1^2(\mathbb C)$ would
induce a real-linear isometric isomorphism
$F_\theta\to\ell_1^2(\mathbb R)$, which is not possible.
\end{proof}
\begin{prop}\label{prop:not-ellinfty}
Let $0<\theta<\pi$, and set
\[
K_\theta=\{e^{is}:-\theta\le s\le\theta\},
\qquad
E_\theta=\operatorname{span}_{\mathbb C}\{1,z\}
\subset C(K_\theta).
\]
Then $E_\theta$ is not complex-linearly isometric to
$\ell_\infty^2$.
\end{prop}

\begin{proof}
Suppose that $T:E_\theta\to\ell_\infty^2$ is a
complex-linear isometry. Write
\[
T1=(u_1,u_2),
\qquad
Tz=(v_1,v_2).
\]
Since $\|1\|_{E_\theta}=\|z\|_{E_\theta}=1$, we have for $j=1,2$
\[
|u_j|\le1,\qquad |v_j|\le1.\]

Note that for every $\zeta\in K_\theta$, $\|1+\overline{\zeta}z\|_{E_\theta}=2.$
Indeed, $|1+\overline{\zeta}\eta|\le2$ for every
$\eta\in K_\theta$, with equality at $\eta=\zeta$.
By linearity and isometry,
\[
\max_{j=1,2}|u_j+\overline{\zeta}v_j|=2.
\]
Hence, for at least one coordinate $j$,
\[
2=|u_j+\overline{\zeta}v_j|
\le |u_j|+|v_j|\le2.
\]
Equality throughout implies that $|u_j|=|v_j|=1$
and that $u_j$ and $\overline{\zeta}v_j$ have the same
argument. Since both have modulus one, they are equal.
Thus we obtain $v_j=\zeta u_j.$ For a fixed coordinate $j$,
this identity determines at most one value of $\zeta$, since
$|u_j|=1$. The two coordinates can therefore account for at most
two points of $K_\theta$, whereas $K_\theta$ contains more than
two points. This is a contradiction.
\end{proof}
\begin{prop}\label{prop:pairwise-nonisometric}
Let $0<\theta,\varphi<\pi$. Then $E_\theta$ and $E_\varphi$
are complex-linearly isometrically isomorphic if and only if
$\theta=\varphi$.
\end{prop}

\begin{proof}
Suppose that $T:E_\theta\to E_\varphi$ is a complex-linear
isometric isomorphism. Write
\[
f=T1=a+bz,\qquad g=Tz=c+dz.
\]
These functions, and hence their coefficients, are fixed
throughout the argument. Since $T$ is an isometry,
\[
\|f\|_{C(K_\varphi)}=\|g\|_{C(K_\varphi)}=1.
\]

For every $\zeta\in K_\theta$,
\[
\|f+\overline{\zeta}g\|_{C(K_\varphi)}
=\|1+\overline{\zeta}z\|_{C(K_\theta)}=2.
\]
By compactness, there is $\eta\in K_\varphi$ such that
\[
2=|f(\eta)+\overline{\zeta}g(\eta)|
\le |f(\eta)|+|g(\eta)|\le2.
\]
Since each of $|f(\eta)|$ and $|g(\eta)|$ is at most one,
equality implies
\[
|f(\eta)|=|g(\eta)|=1.
\]
Equality in the triangle inequality also gives
\[
g(\eta)=\zeta f(\eta).
\]

For these fixed functions, put
\[
S=\{\eta\in K_\varphi:|f(\eta)|=|g(\eta)|=1\}.
\]
The preceding argument shows that
\[
K_\theta\subseteq
\left\{\frac{g(\eta)}{f(\eta)}:\eta\in S\right\}.
\]
The quotients are well defined because $|f(\eta)|=1$ on $S$.
If $S$ were finite, the set on the right would be finite,
which is impossible since $K_\theta$ is infinite.
Thus $S$ is infinite.

By definition of $S$, every $\eta\in S$ satisfies
$|f(\eta)|=|g(\eta)|=1$. Since $z(\eta)=\eta$, we have
\[
f(\eta)=a+b\eta,\qquad g(\eta)=c+d\eta.
\]
Consequently,
\[
|a+b\eta|=|c+d\eta|=1
\qquad\text{for all }\eta\in S.
\]

For $\eta\in\mathbb T$,
\[
|a+b\eta|^2
=|a|^2+|b|^2
+2\operatorname{Re}(a\overline{b}\,\overline{\eta}).
\]
If $a\overline{b}\ne0$, the equation $|a+b\eta|=1$ has at most two solutions. Since this
equation holds on the infinite set $S$, we obtain
$a\overline{b}=0$. Similarly, $c\overline{d}=0$.

It follows that
\[
a=0\ \text{or}\ b=0,
\qquad
c=0\ \text{or}\ d=0.
\]
Thus each of $f,g$ is either a constant or a multiple of $z$.
Neither function is zero, since both have norm one.
Moreover, $f,g$ are linearly independent because $T$
is injective and $1,z$ are linearly independent.
They therefore cannot both be constants or both be
multiples of $z$.

If $b=0$, then $f=a$ is a nonzero constant, and independence
forces $g=dz$, so $c=0$. The norm-one condition gives
$|a|=|d|=1$.
If $a=0$, then $f=bz$, and independence forces $g=c$
to be a nonzero constant, so $d=0$. In this case
$|b|=|c|=1$. Hence the map
\[
h(\eta):=\frac{g(\eta)}{f(\eta)},\qquad\text{for all }\eta\in\mathbb T,
\]
has one of the forms
\[
h(\eta)=\frac{\mu}{\lambda}\eta
\qquad\text{or}\qquad
h(\eta)=\frac{\mu}{\lambda}\overline{\eta},
\] for some $\mu,\lambda\in\mathbb T$. Now for $\zeta,\eta\in\mathbb T$, we have
\[
|1+\overline{\zeta}\eta|=2
\quad\Longleftrightarrow\quad \eta=\zeta.
\]
Similarly, since $|f(\eta)|=|g(\eta)|=1$,
\[
|f(\eta)+\overline{\zeta}g(\eta)|=2
\quad\Longleftrightarrow\quad h(\eta)=\zeta.
\]
Using compactness and the isometry of $T$, we therefore obtain,
for every $\zeta\in\mathbb T$,
\[
\begin{aligned}
\zeta\in K_\theta
&\Longleftrightarrow
\|1+\overline{\zeta}z\|_{C(K_\theta)}=2\\
&\Longleftrightarrow
\|f+\overline{\zeta}g\|_{C(K_\varphi)}=2\\
&\Longleftrightarrow
\zeta\in h(K_\varphi).
\end{aligned}
\]
Thus $K_\theta=h(K_\varphi)$. Since $h$ is a rotation
or a rotation followed by conjugation, it preserves arc
length. Therefore $2\theta=2\varphi$, so $\theta=\varphi$.
\end{proof}

\appendix
\section{Complexification of a real Hilbert space}\label{app:complexification}
 The complexification of a real Hilbert space $(H,\ \langle .,.\rangle_{H})$ is defined to be the vector space
\[
   \HC:=H+iH=\{x+iy:x,y\in H\},
\]
where the decomposition $x+iy$ is unique. Addition is defined componentwise,
and, for $a,b\in\mathbb R$, complex scalar multiplication is given by
\[
 (a+ib)(x+iy)=(ax-by)+i(bx+ay).
\]
The canonical complex inner product is
\begin{equation}\label{eq:inner}
\langle x+iy,u+iv\rangle_{\HC}
   =(x,u)_H+(y,v)_H
       +i\bigl[(y,u)_H-(x,v)_H\bigr].
\end{equation}
In particular, symmetry of the original real inner product implies
\[
 \|x+iy\|_{\HC}^2
 =\langle x+iy,x+iy\rangle_{\HC}
 =\|x\|_H^2+\|y\|_H^2.
\]
It follows that $\HC$ is a complex Hilbert space.
The map $x\mapsto x+i0$ identifies $H$ isometrically with a closed real
subspace of $\HC$. The canonical conjugation is
\[
 \kappa:\HC\longrightarrow\HC,\qquad \kappa(x+iy)=x-iy;
 \qquad \kappa^2=I,\qquad \langle\kappa u,\kappa v\rangle=\overline{\langle u,v\rangle}, \ \text{for all} u, v\in H_{\mathbb{C}}.
\]
It is a conjugate-linear isometric involution and its fixed-point space is $H$.

Let $T\in\BH$ be a bounded \emph{real-linear} operator. Its complexification
is defined by
\begin{equation}\label{eq:extension}
 \quad T_{\mathbb C}(x+iy):=Tx+iTy,\qquad\text{for all}\ x,y\in H.\quad
\end{equation}
We write $T\geq0$ when $T=T^*$ and $\langle Tx,x\rangle_H\geq0$ for all $x\in H$.
The following proposition is easy to check.
\begin{prop}\label{prop:complexification}
The operator $T_{\mathbb C}\colon\HC\to\HC$ is the unique complex-linear
extension of $T$. Moreover,
\begin{align}
 \|T_{\mathbb C}\|_{\BHC}&=\|T\|_{\BH},
 \label{eq:norm}\\
 (T_{\mathbb C})^*&=(T^*)_{\mathbb C},
 \label{eq:adjoint}\\
 T_{\mathbb C}\kappa&=\kappa T_{\mathbb C}.
 \label{eq:conjugation}
\end{align}
If $S\in\BH$, then
\begin{equation}\label{eq:product}
 (ST)_{\mathbb C}=S_{\mathbb C}T_{\mathbb C}.
\end{equation}
In particular, $T$ is self-adjoint, normal,
contraction, isometry, orthogonal and positive if and only if $T_{\mathbb{C}}$ is
self-adjoint, normal, contraction, isometry, unitary and positive, respectively.
\end{prop}
\begin{prop}\label{prop:abs}
Let $H$ be a real Hilbert space and $T\in\mathcal B(H)$. Then there is a unique
operator $|T|\in\mathcal B(H)$ with
\[
 |T|\ge0,\qquad |T|^2=T^*T .
\]
Moreover, $|T|_{\mathbb C}=|T_{\mathbb C}|$, where $|T_{\mathbb C}|=(T_{\mathbb C}^*T_{\mathbb C})^{1/2}$
is the usual absolute value of $T_{\mathbb C}.$
\end{prop}

\begin{proof}
Put $A=T^*T$. Then by Proposition~\ref{prop:complexification}, $(A_{\mathbb C})^{1/2}=|T_{\mathbb C}|$.

Note that
$\langle\kappa( A_{\mathbb C})^{1/2}\kappa v,v\rangle=\overline{\langle (A_{\mathbb C})^{1/2}\kappa v,\kappa v\rangle}\ge0$. Also,
$\bigl(\kappa(A_{\mathbb C})^{1/2}\kappa\bigr)^2=\kappa A_{\mathbb C}\kappa=A_{\mathbb C}$ by \eqref{eq:conjugation}. By uniqueness of the
positive square root on $H_{\mathbb C}$, we get $\kappa( A_{\mathbb C})^{1/2}\kappa=(A_{\mathbb C})^{1/2}.$

Put $B=(A_{\mathbb C})^{\frac{1}{2}}.$ For $x\in H$, $\kappa Bx=B\kappa x=Bx$, so $Bx\in H$. Thus
$S:=B|_H\in\mathcal B(H)$ and $B=S_{\mathbb C}$. By \eqref{eq:adjoint},
$(S^*)_{\mathbb C}=B^*=B=S_{\mathbb C}$, so $S=S^*$. Also
$S\ge0$ since $S_{\mathbb C}=B\ge0$. By \eqref{eq:product}, $(S^2)_{\mathbb C}=B^2=A_{\mathbb C}$,
so $S^2=A$. Set $|T|:=S$. Hence the square root in the real case exists and the uniqueness follows from the complex case.

Finally, by construction $|T|_{\mathbb C}=S_{\mathbb C}=B=|T_{\mathbb C}|$.
\end{proof}
The same argument, starting with any positive $P\in\mathcal B(H)$ in place of $A=T^*T$, gives its unique positive square root and $(P^{1/2})_{\mathbb C}=(P_{\mathbb C})^{1/2}$.
\section{Realification of a complex Hilbert space}\label{app:realification}
Let $(H,\langle\cdot,\cdot\rangle)$ be a complex Hilbert space. Its
\emph{realification} $H_{\mathbb R}$ is the same set with scalar multiplication
restricted to $\mathbb R$, equipped with the real inner product
\[
 (x,y)_{\mathbb R}:=\operatorname{Re}\langle x,y\rangle .
\]
Then $(x,x)_{\mathbb R}=\|x\|^2$, so $H_{\mathbb R}$ is a real Hilbert space
with the same norm. Let $J:H_{\mathbb R}\to H_{\mathbb R}$, $Jx=ix$. Then
\[
 J^2=-I,\qquad J^*=-J,\qquad J \text{ is orthogonal.}
\]
The complex inner product is recovered by
$\langle x,y\rangle=(x,y)_{\mathbb R}+i\,(x,Jy)_{\mathbb R}$.
Every complex-linear $T\in\mathcal B(H)$ is real-linear on $H_{\mathbb R}$;
write $T_{\mathbb R}$ for this operator.

\begin{prop}\label{prop:realification}
Let $T,S\in\mathcal B(H)$ be complex-linear. Then
\begin{align}
 \|T_{\mathbb R}\|&=\|T\|,\label{eq:rnorm}\\
 (T_{\mathbb R})^*&=(T^*)_{\mathbb R},\label{eq:radjoint}\\
 (ST)_{\mathbb R}&=S_{\mathbb R}T_{\mathbb R},\qquad
 T_{\mathbb R}J=JT_{\mathbb R}.\label{eq:rproduct}
\end{align}
Conversely, a real-linear operator on $H_{\mathbb R}$ is of the form
$T_{\mathbb R}$ for a complex-linear operator $T$ if and only if it commutes
with $J$. In particular, $T$ is self-adjoint, normal, a contraction, an
isometry, unitary or positive if and only if $T_{\mathbb R}$ is
self-adjoint, normal, a contraction, an isometry, orthogonal or positive,
respectively. Moreover, for a real-linear $A$ on $H_{\mathbb R}$, the operator
\[
 \tfrac12\,(A-JAJ)
\]
is complex-linear, and it has norm at most $\|A\|$.
\end{prop}
For $P\ge0$, Proposition~\ref{prop:realification} and uniqueness of the positive square root give $(P^{1/2})_{\mathbb R}=(P_{\mathbb R})^{1/2}$. This is the square-root identity used in the complex case of Lemma~\ref{lem:completion}.
\paragraph{Acknowledgement.}
The named author thanks the support of the Prime Minister Early Career Research Grant
ANRF/\allowbreak ECRG/\allowbreak 2024/\allowbreak 000699/\allowbreak PMS and ANRF-MATRICS grant ANRF/\allowbreak ARGM/\allowbreak 2025/\allowbreak 000895/\allowbreak MTR.

\paragraph{Declaration of AI assistance.}
The author used ChatGPT 6.1 Sol (OpenAI) for substantive mathematical
assistance, including proposing circular-arc examples and developing
proof drafts, and for language and LaTeX editing. The author posed
the problem, directed the search using the results in
\cite{ARFJS95} and his earlier work \cite{Ray20}, provided geometric
intuitions, critically examined and revised
the proposed arguments, and developed alternative proofs, selecting
those he considered most natural. The author takes full
responsibility for the final manuscript.

\begin{flushleft}
\small
\textit{The Institute of Mathematical Sciences, 4th Cross Street,
CIT Campus, Tharamani, Chennai, Tamil Nadu 600113, India, and}\\[0.5em]
\textit{Homi Bhabha National Institute, Training School Complex,
Anushakti Nagar, Mumbai 400094, India}\\[0.5em]
\textit{Email address:} \href{mailto:samya@imsc.res.in}{\texttt{samya@imsc.res.in}}
\end{flushleft}
\end{document}